\documentclass[twoside,a4paper,12pt,reqno]{amsart}

\usepackage[OT2,T1]{fontenc}

\newcommand\textcyr[1]{{\fontencoding{OT2}\selectfont #1}}

\usepackage{amsmath,amssymb,comment,epigraph,upgreek}

\DeclareSymbolFont{euleroperators}{U}{eur}{m}{n}
\SetSymbolFont{euleroperators}{bold}{U}{eur}{b}{n}

\makeatletter
\renewcommand{\operator@font}{\mathgroup\symeuleroperators}
\makeatother

\makeatletter
\def\@seccntformat#1{%
  \protect\textup{\protect\@secnumfont
    \ifnum\pdfstrcmp{subsection}{#1}=0 \bfseries\fi
    \csname the#1\endcsname
    \protect\@secnumpunct
  }%
}
\makeatother

\newtheorem{thm}[equation]{Theorem}
\newtheorem{clm}[equation]{Claim}

\newtheorem{cor}[equation]{Corollary}
\newtheorem{lem}[equation]{Lemma}
\newtheorem{prp}[equation]{Proposition}

\theoremstyle{definition}

\theoremstyle{remark}
\newtheorem{rem}[equation]{Remark}

\newcommand{\thmref}[1]{Theorem~\ref{#1}}
\newcommand{\prpref}[1]{Proposition~\ref{#1}}
\newcommand{\lemref}[1]{Lemma~\ref{#1}}

\newcommand{\dfnref}[1]{Definition~\ref{#1}}
\newcommand{\remref}[1]{Remark~\ref{#1}}

\newcommand{\secref}[1]{Section~\ref{#1}}

\newcommand{\clmref}[1]{Claim~\ref{#1}}

\DeclareMathOperator{\bnd}{bnd}
\DeclareMathOperator{\diag}{diag}
\DeclareMathOperator{\fun}{fun}
\DeclareMathOperator{\gr}{gr}
\DeclareMathOperator{\sgr}{sgr}
\DeclareMathOperator{\supp}{supp}
\DeclareMathOperator{\Unif}{Unif}

\newcommand{\al}{\alpha}

\newcommand{\de}{\delta}

\newcommand{\Eg}{\mathfrak E}
\newcommand{\eps}{\varepsilon}

\newcommand{\gb}{\mathbf{g}}

\newcommand{\ka}{\varkappa}

\newcommand{\la}{\lambda}
\newcommand{\Lc}{\mathcal L}

\newcommand{\N}{{\boldsymbol{\#}}}
\newcommand{\nnu}{\boldsymbol{\upnu}}

\newcommand{\ol}{\overline}
\newcommand{\om}{\omega}

\newcommand{\Pb}{\mathbf P}
\newcommand{\ph}{\varphi}
\newcommand{\Ps}{\p_\mu G}
\newcommand{\p}{\partial}

\newcommand{\Sc}{\mathcal S}
\newcommand{\sd}{\rightthreetimes}
\newcommand{\si}{\sigma}
\newcommand{\sm}{\setminus}

\newcommand{\tos}{\mathop{\,\relbar\joinrel\rightsquigarrow\,}}
\newcommand{\tha}{\theta}

\newcommand{\upth}{\uptheta}

\newcommand{\wh}{\widehat}

\newcommand{\wt}{\widetilde}

\newcommand{\ZZ}{\mathbb Z}

\begin{document}

\title[Random walks with infinite entropy]{Liouville property and Poisson boundary of random walks with infinite entropy: what's amiss?}

\author{Vadim A. Kaimanovich}

\address{Department of Mathematics and Statistics, University of Ottawa, Canada}

\email{vadim.kaimanovich@gmail.com}

\dedicatory{To the memory of my teacher Anatoly Moiseevich Vershik}

\begin{abstract}
We discuss the qualitatively new properties of random walks on groups that arise in the situation when the entropy of the step distribution is infinite.
\end{abstract}

\subjclass{28D20, 37A50, 43A07, 60J50}
\keywords{random walk, asymptotic entropy, Liouville property, Poisson boundary.}

\maketitle

\thispagestyle{empty}

\section*{Introduction}

{\bf 1.} Endowing a countable group $G$ with a probability measure $\mu$ rigs the group with an additional structure in a sense akin to a Riemannian metric on a smooth manifold. In both situations this additional structure gives rise to a Markov dynamics on the original state space: this is a Brownian motion in the manifold case and a random walk in the group case.

In the framework of the general theory of Markov processes one can then define the \emph{Poisson boundary} of the state space. In the probabilistic language this is a measure space that describes all stochastically significant behaviour of the sample paths of the respective Markov process at infinity. In ergodic terms this is the space of ergodic components of the time shift on the path space, whereas in analytic terms this is the spectrum of the commutative Banach algebra of bounded harmonic functions.

It is then natural to ask about a complete description (identification) of the Poisson boundary in terms of the ``observed'' limit behaviour of Markov sample paths. This behaviour can be characterized by their analytic, geometric, algebraic, or combinatorial properties. In the particular case when no limit behaviour is observed at all, already the question about triviality of the Poisson boundary, i.e., about absence of non-constant bounded harmonic functions (the Liouville property), or, equivalently, about ergodicity of the time shift on the path space, is highly non-trivial (in the same way as the question about ergodicity of many other dynamical systems).

From a probabilistic point of view the Liouville property is equivalent to asymptotic vanishing of the dependence of the one-dimensional distributions of the Markov process on its initial distribution as time goes to infinity. This is why for random walks on groups this property can also be formulated in an entirely analytic language, without relying on any probabilistic interpretations. Namely, an irreducible probability measure $\mu$ on a group $G$ is Liouville if and only if the \emph{means} on $G$ that arise by passing to a Banach limit of the sequence of the convolution powers~$\mu^{*t}$ are left invariant
(see~\S~\ref{sec:im} as well as \remref{rem:inv} and \remref{rem:im2}).

\medskip

{\bf 2.} The fact that random walks on groups are homogeneous not only in time, but also in space, allows one to define the \emph{asymptotic entropy}~$h(\mu)$,
which makes the aforementioned question about the Liouville property entirely quantitative: the Poisson boundary $\Ps$ of a random walk $(G,\mu)$ is trivial if and only if $h(\mu)=0$ (for the sake of simplicity we assume, throughout the Introduction, that the Poisson boundary is endowed with the \emph{primary harmonic measure} $\nu$ that corresponds to the initial distribution concentrated at the group identity, see \S~\ref{sec:pois}). There is an analogous criterion (based on the asymptotic entropy of conditional random walks) for the general problem of identification of the Poisson boundary as well. Nonetheless, these criteria only work under the \emph{finiteness assumption} on the entropy $H(\mu)$ of the step distribution $\mu$.

The asymptotic entropy
$$
h(\mu)=\lim_{t\to\infty} \frac{H(\mu^{*t})}{t}
$$
(where $\mu^{*t}$ denotes the $t$-fold convolution of the measure $\mu$ with itself) was introduced by A.~Avez in 1972 \cite{Avez72} (and a bit later, independently, by A.~M.~Vershik), and the closely related (as became clear later) differential entropy of the boundary had been used by H.~Furstenberg even earlier~\cite{Furstenberg63a}. However, the full development of the entropy theory was achieved only on the basis of ergodic methods related to an analogy between the asymptotic entropy of random walks on groups and the entropy of dynamical systems. This program was proposed by A.~M.~Vershik in the 1970s (see \cite[p.~60]{Vershik00r}) and implemented in his joint works with the author \cite{Vershik-Kaimanovich79r,Kaimanovich-Vershik83}, as well as independently (and using a somewhat different approach) by Y.~Derriennic \cite{Derriennic80}. Subsequently, the entropy theory of random walks was extended by the author to a much broader class of ``stochastically homogeneous'' Markov chains (see  \remref{rem:gen}).

A rather unexpected consequence of the entropy theory was the appearance of the finiteness condition for the entropy of the step distribution $\mu$ in a situation that at first glance is completely unrelated to entropy. The entropy of the reflected measure $\check\mu(g)=\mu(g^{-1})$ obviously coincides with the entropy of the original measure $\mu$, and therefore their asymptotic entropies also coincide. Thus, both measures $\mu$ and $\check\mu$ are either Liouville or non-Liouville simultaneously, i.e., left invariance of the corresponding limit means on the group is automatically equivalent to their right invariance. On the other hand, for measures with infinite entropy, this is not necessarily the case: already in 1983 the author constructed an example of a measure $\mu$ with infinite entropy that is ``one-sided Liouville'' (it is Liouville itself, while the reflected measure $\check\mu$ is not Liouville).

\medskip

{\bf 3.} The aim of the present paper is to demonstrate the critical role of entropy finiteness in yet another situation: for random walks on products of several groups.

Let us consider the product $\wt G=G\times G'$ of two (for simplicity) groups $G$ and $G'$, with the random walk defined by a step distribution $\wt\mu$. The projections of this walk onto each coordinate are the random walks on $G$ and $G'$ determined by the respective marginal distributions $\mu$ and~$\mu'$ of the measure $\wt\mu$. It is natural to ask about the relationship between the Liouville properties of these three measures, or, more generally, to ask how the Poisson boundaries of these three walks are related.

The simplest way to obtain a measure with the given marginals $\mu$ and $\mu'$ is to take their product $\wt\mu=\mu\otimes\mu'$. One can also consider a convex combination
$$
\wt\mu=\al\mu\otimes\de_{e'} + (1-\al)\de_e\otimes\mu' \;,
$$
where $e$ and $e'$ are the identities of the respective groups. In the latter case the marginals are not the measures $\mu,\mu'$ themselves, but their ``lazy'' modifications $\al\mu+(1-\al)\de_e$ and $(1-\al)\mu'+\al\de_{e'}$, which does not really change anything, as the Poisson boundaries of the lazy measures are the same as those of the respective original measures $\mu,\mu'$. In both cases the Poisson boundary $\p_{\wt\mu}\wt G$ is the product of the Poisson boundaries of the quotient groups $\Ps$ and $\p_{\mu'}G'$ as measure spaces (i.e., the primary harmonic measure of the random walk on $\wt G$ is a product measure), and, in particular, the measure $\wt\mu$ is Liouville if and only if both measures $\mu$ and $\mu'$ are simultaneously Liouville.

These definitions are not specific to random walks, and they can be applied to arbitrary Markov chains (with a formal replacement of step distributions with Markov operators). The resulting Markov chains are called, respectively, \emph{direct and Cartesian products of the original chains,} and the proofs of the aforementioned properties boil down to verifying that the product of two minimal positive harmonic functions is also minimal with respect to the product operator \cite[p.~146]{Kaimanovich92}. Conversely, the statements about the factorization of minimal positive harmonic functions of the product operator (S.~A.~Molchanov \cite{Molchanov67r} for direct products, M.~Picardello and W.~Woess \cite{Picardello-Woess92} for Cartesian products) are much less obvious (in particular, because they require involvement of $\la$-harmonic functions with $\la\neq 1$), and their proofs rely on the Martin theory.

In the context of arbitrary Markov chains, there are no other universal methods for obtaining a Markov chain on the product space whose marginals would be given chains on the spaces-multipliers, but this is not the case for random walks on groups. Here we are talking precisely about the aforementioned random walks whose step distribution $\wt\mu$ has given marginal distributions $\mu,\mu'$. Although this formulation is very natural, such walks have not been subject to a systematic study, even though ``multicomponent'' Poisson boundaries did arise, for example, in the work of the author \cite{Kaimanovich00a} for discrete subgroups in the product of several simple Lie groups or in the work of S.~Brofferio and B.~Shapira \cite{Brofferio-Schapira11} for groups of rational matrices. Author's interest in this subject was sparked by a question of R.~I.~Grigorchuk about random walks on the product of several copies of the same group, which arise when considering self-similar groups, as such products have a finite index in the corresponding permutational wreath products \cite{Grigorchuk-Kaimanovich23}.

\medskip

{\bf 4.} It is not difficult to verify (\prpref{prp:2} and \remref{rem:gen}) that under the condition of finiteness of the entropy of the measure $\wt\mu$, the Poisson boundary $\p_{\wt\mu}\wt G$ is fully described by the behavior of the quotient walks on the groups-multipliers, i.e., the mapping $\p_{\wt\mu}\wt G\to \Ps\times\p_{\mu'} G'$ is an isomorphism of measure spaces. It should be emphasized that although the marginals of the primary harmonic measure $\wt\nu$ are the respective primary harmonic measures $\nu,\nu'$, this by no means implies that $\wt\nu$ is their product. The question of describing the measure $\wt\nu$ seems to be quite interesting, and it remains completely open even in the simplest situations, for example, for the measure $\wt\mu=\frac12(\mu\otimes\mu + \diag\mu)$ in the case when $G=G'$ is a free group, $\mu$ is the uniform distribution on the (symmetric) set of free generators, and $\diag\mu$ is the corresponding measure on the diagonal of the product $G\times G$. Note that the spatial homogeneity is essential in defining a random walk on $G\times G$ with the transition measure $\diag\mu$, and such a ``diagonal'' Markov operator cannot be defined for an arbitrary Markov chain.

In particular, if the entropy $H(\wt\mu)$ of the measure $\wt\mu$ is finite, then the Liouville property for this measure is equivalent to its marginals $\mu$ and $\mu'$ being Liouville simultaneously (\prpref{prp:2}). We show that if the entropy finiteness condition is dropped, this is no longer necessarily the case, by explicitly constructing the corresponding counterexample (\thmref{thm:3}).

This construction is based on using, in a slightly modified form, the aforementioned example of a one-sided Liouville measure $\mu$ on the wreath product $\Lc=\ZZ\wr\ZZ_2$ (the one-dimensional lamplighter group), whose projection onto $\ZZ$ has a non-zero mean. The idea is quite transparent and consists in taking the ``semidiagonal'' subgroup $\wt\Lc_0$ with a common $\ZZ$-component in the product $\wt\Lc=\Lc\times\Lc$, and then considering a measure $\wt\mu$ on it whose marginals coincide with $\mu$, and thus are Liouville. On the other hand, it is easy to choose $\wt\mu$ in such a way that it is ``almost'' (but not completely) diagonal in the sense that its image on $\Lc$ under the mapping $(n,\ph,\ph')\to (n,\ph-\ph')$ has a non-singleton finite support and is not Liouville, thus ensuring that the measure $\wt\mu$ is not Liouville either.

\medskip

{\bf 5.} With rare exceptions (for example, see the ``trunk'' criterion of A.~Erschler and the author \cite[Proposition 3.8]{Erschler-Kaimanovich23}), the entropy theory (more precisely, the criterion from \cite{Kaimanovich00a}) remains the only means of identifying a non-trivial Poisson boundary of random walks on countable groups. Very recently, the \emph{Ultima Thule} of the entropy theory has been reached for two classes of groups, serving as the main examples of this kind: Gromov hyperbolic groups (and their various modifications) in the work of K.~Chawla, B.~Forghani, J.~Frisch, and G.~Tiozzo \cite{Chawla-Forghani-Frisch-Tiozzo22p}, and for wreath products in the work of J.~Frisch and E.~Silva \cite{Frisch-Silva23p}. Namely, it was shown that in both cases convergence to the ``natural'' boundary (the hyperbolic one in the first case and the space of infinite configurations in the second case) indeed fully describes the Poisson boundary without any additional conditions on the step distribution, other than finiteness of its entropy. It is worth noting, though, that in the hyperbolic case convergence to the boundary does not require any assumptions about the step distribution at all (except for the purely geometric condition of non-elementarity), and the problem of identifying the Poisson boundary, for example, for a random walk on a free group, posed by A.~M.~Vershik and the author 40 years ago \cite[p.~485]{Kaimanovich-Vershik83}, still remains completely open (it remains open even for free semigroups, in the absence of cancellations, see the description of the current state of affairs in \cite{Chawla-Forghani-Frisch-Tiozzo22p}). The situation with wreath products is somewhat more complicated, and we refrain from discussing the preceding very interesting results in this direction, referring the reader to the detailed review given in \cite{Frisch-Silva23p}.

In light of this, it should be noted that the example we constructed is, to the best of author's knowledge, the first one in which the Poisson boundary of a random walk with infinite entropy (in our example described by nontrivial \emph{joint} behavior of pairs of configurations) would be qualitatively different from the Poisson boundaries of random walks with finite entropy (described in our case by \emph{separate} behavior of the marginal walks).

The recent breakthrough of J.~Frisch, Y.~Hartman, O.~Tamuz, and P.~Vahidi Ferdowsi \cite{Frisch-Hartman-Tamuz-Ferdowsi19} led to a complete description of the groups for which both Liouville and non-Liouville measures coexist. This is precisely the class of amenable groups that are not hyper-FC-central (do not coincide with their hyper-FC-center), or, equivalently, the amenable groups with an ICC quotient group (i.e., one in which all nontrivial conjugacy classes are infinite). Soon thereafter, A.~V.~Alpeev \cite{Alpeev21p} showed that one-sided Liouville measures exist on \emph{any} group from this class. After the author had informed him of the main result of this work, A.~V.~Alpeev \cite{Alpeev23p} proved that an example with the same properties can be constructed for the product of \emph{arbitrary} groups from this class, and moreover, all measures involved in the construction can be chosen symmetric and irreducible.

Finally, it should be noted that it would be wrong to think that ``all'' measures with infinite entropy are ``pathological.'' There is a very general construction described in a joint work with B.~Forghani \cite{Forghani-Kaimanovich23p} that allows, starting from a certain fixed step distribution, to build numerous other measures (in particular, those with infinite entropy) with the same Poisson boundary.

\medskip

{\bf 6.} The paper is organized as follows. We begin with recalling the main definitions and results used throughout, while formulating and proving necessary auxiliary statements along the way. In the first two sections we discuss the Poisson boundary (\secref{sec:1}) and the entropy theory of random walks (\secref{sec:e}). Then we talk about wreath products of groups (\secref{sec:3}) and random walks on them (\secref{sec:4}). Here we describe the random walks on the groups $\Lc$ and $\wt\Lc_0$ (mentioned in {\bf n$^\text{o}$4} above), and establish conditions under which they are not Liouville (\clmref{clm:l1}, \clmref{clm:reflected}, and \clmref{clm:wtmu}). Finally, in \secref{sec:inf} we first prove the main technical \lemref{lem:lem}: the equivalence of the Liouville property and infiniteness of the step distribution entropy for the considered class of measures on the group $\Lc$ (which, for this class of measures, is also equivalent to the one-sided Liouville property, \thmref{thm:2}). Then, we obtain our main result: the existence of a non-Liouville measure (as noted in {\bf n$^\text{o}$4}, it must have infinite entropy) on the product of two groups, both of whose marginals are Liouville (\thmref{thm:3}). Paying tribute to the constant interest of Anatoly Moiseevich Vershik in the history of mathematics and its reflection in modern times (see \cite{Vershik00r,Vershik10ar,Vershik20r} regarding the topics discussed here), we conclude the article with a brief historical commentary on the origin of the main concepts featured in its title: Liouville's property and Shannon's entropy.

\textbf{Acknowledgements.} I am grateful to R.~I.~Grigorchuk for numerous discussions on self-similar groups and related probabilistic problems which served as a stimulus for considering random walks on group products, and to A.~Erschler for discussing one-sided invariant means (see \remref{rem:im2}). In conclusion, I express my deep gratitude to my teacher Anatoly Moiseevich Vershik, who introduced me to the world of mathematics. The topics of this article also go back to his questions and ideas.

\section{Poisson boundary and Liouville property} \label{sec:1}

\begin{flushright}
\begin{tabular}{l}
\textcyr{Поэзия, должно быть, состоит}
\\
\textcyr{в от{с}ут{с}твии отчетливой границы}
\\
\hfill\textit{\textcyr{И.~Бродский}}. \textit{Post aetatem nostram}
\end{tabular}
\end{flushright}

\subsection{}

The \textsf{random walk} $(G,\mu)$ on a countable group $G$ determined by a probability measure~$\mu$ is the Markov chain with the state space $G$ and transition probabilities
$$
p(g,g') = \mu(g^{-1} g')
$$
invariant under the left action of the group $G$ on itself. In other words, the Markov transitions
\begin{equation} \label{eq:trans}
g \tos^{h\sim\mu} g'=gh
\end{equation}
of the random walk $(G,\mu)$ consist in right multiplication by a $\mu$-distributed random \textsf{increment}~$h$. The resulting \textsf{Markov operator} $P=P_\mu$ acts on functions on the group $G$ by averaging as
\begin{equation} \label{eq:p}
Pf (g) = \sum_h \mu(h) f(gh) \;,
\end{equation}
and its dual operator
\begin{equation} \label{eq:pp}
\tha \mapsto \tha P = \tha * \mu
\end{equation}
acts on measures by applying the Markov transitions \eqref{eq:trans}, i.e., as the right convolution with the measure $\mu$.

Any initial distribution $\tha$ (not necessarily a probability one) determines the associated \textsf{Markov measure} $\Pb_\tha$ on the space $G^{\ZZ_+}$ of \textsf{sample paths} $\gb=(g_0,g_1,g_2,\dots)$, where
$$
g_t = g_0 h_1 h_2 \dots h_t \qquad \forall\,t\in\ZZ_+ \;,
$$
and $(h_t)$ is a sequence of independent $\mu$-distributed increments of the random walk. Thus, the one-dimensional distribution of the measure $\Pb_\tha$ at time $t$ is $\tha P^n=\tha*\mu^{*t}$, where $\mu^{*t}$ denotes the $t$-fold convolution of the measure $\mu$ with itself. Since the space of sample paths is equipped with a natural coordinate-wise action of the group $G$
\begin{equation} \label{eq:action}
g\gb = (gg_0, gg_1, gg_2, \dots) \qquad \forall\,g\in G,\; \gb=(g_0,g_1,g_2,\dots)\in G^{\ZZ_+} \;,
\end{equation}
and the mapping $\tha\mapsto\Pb_\tha$ is equivariant due to the spatial homogeneity of the transition probabilities, the measure $\Pb_\tha$ is the convolution of the initial distribution $\tha$ with the measure $\Pb=\Pb_{\de_e}$ corresponding to the initial distribution concentrated at the identity $e$ of the group~$G$, namely,
$$
\Pb_\tha = \tha * \Pb = \sum_g \mu(g) g\Pb \;,
$$
where $g\Pb=\Pb_g=\Pb_{\de_g}$ is the measure on the space of sample paths issued from a point $g\in G$. The \textsf{counting measure} $\N$ on the group $G$ is preserved by the transition operator \eqref{eq:pp} of the random walk, and therefore the corresponding $\si$-finite measure $\Pb_\N$ on the space of sample paths is invariant under the \textsf{shift}
\begin{equation} \label{eq:shift}
T: (g_0,g_1,g_2,\dots) \mapsto (g_1,g_2,g_3,\dots) \;.
\end{equation}

\subsection{} \label{sec:pois}

Two sample paths $\gb=(g_0,g_1,\dots)$ and $\gb'=(g'_0,g'_1,\dots)$ are called \textsf{(asynchronously) asymptotically equivalent} if they coincide from some moment onwards (specific to each sample path), i.e., if there exist integers $t,t'\geq 0$ such that $g_{t+i}=g'_{t'+i}$ for all $i\geq 0$. In other words, this is the \textsf{orbital equivalence relation} of the shift \eqref{eq:shift}: sample paths $\gb$ and $\gb'$ are equivalent if $T^t\gb=T^{t'}\gb'$ for some $t,t'$. The measurable hull $\Eg$ of this equivalence relation, i.e., the $\si$-algebra of all measurable subsets of the space $(G^{\ZZ_+},\Pb_\N)$ that are unions of equivalence classes ($\Pb_\N$ -- mod~0), is called the \textsf{exit $\sigma$-algebra} of the random walk $(G,\mu)$. Equivalently, $\Eg$ is the completion of the $\si$-algebra of $T$-invariant subsets of the space of sample paths $G^{\ZZ_+}$ with respect to the measure $\Pb_\N$. The arising quotient of the space of sample paths $(G^{\ZZ_+},\Pb_\N)$, i.e., the space of ergodic components of the shift $T$ with respect to the invariant measure $\Pb_\N$, is called the \textsf{Poisson boundary} $\Ps$ of the random walk $(G,\mu)$, and we denote by
$$
\bnd: G^{\ZZ_+} \to \Ps
$$
the corresponding canonical projection. The \textsf{harmonic measure} of an initial distribution $\tha$ is the image
$$
\nu_\tha=\bnd\Pb_\tha
$$
(which is well-defined because the corresponding measure $\Pb_\tha$ on the space of sample paths is absolutely continuous with respect to $\Pb_\N$). We denote by $\nnu$ the \textsf{harmonic measure class} on $\Ps$, i.e., the common class of the harmonic measures $\nu_\tha$ corresponding to the initial distributions $\tha$ with full support $\supp\tha=G$. Note that the image $\bnd\Pb_\N$ of the measure~$\Pb_\N$ itself is trivially infinite, which is why we have to introduce the harmonic measure class (alternatively, one could fix a reference measure from this class). The harmonic measure $\nu_\tha$ determined by any initial probability distribution $\tha$ is, of course, a probability one, and it is dominated by the class $\nnu$, i.e., it is absolutely continuous with respect to $\nnu$.

Since the action \eqref{eq:action} of the group $G$ on the space of sample paths $G^{\ZZ_+}$ commutes with the shift \eqref{eq:shift}, it descends to the Poisson boundary, and the harmonic measure class $\nnu$ is quasi-invariant under this action. We denote by
\begin{equation} \label{eq:nu}
\nu=\bnd\Pb
\end{equation}
the \textsf{primary harmonic measure} determined by the single-point initial distribution concentrated at the identity of the group. Then
$$
\nu_g=\bnd\Pb_g=g\nu
$$
for any initial point $g\in G$, and $\nu_\tha=\tha*\nu$ for any initial distribution $\tha$. The primary harmonic measure $\nu$ is \textsf{$\mu$-stationary} in the sense that
\begin{equation} \label{eq:stat}
\nu = \bnd \Pb = \bnd (T\Pb) = \bnd \Pb_\mu = \sum_g \mu(g) g\nu = \mu*\nu \;.
\end{equation}
As we have already noted, the measure $\nu$ is dominated by the harmonic measure class $\nnu$. If the step distribution $\mu$ is \textsf{non-degenerate} in the sense that the semigroup $\sgr\mu$ generated by its support $\supp\mu$ coincides with the whole group $G$, then, as follows from \eqref{eq:stat}, the measure $\nu$ is quasi-invariant and belongs to the class $\nnu$. In general, even if the measure $\mu$ is \textsf{irreducible} (the group $\gr\mu$ generated by its support coincides with $G$), this domination can be strict (which leads to non-trivial questions, see the discussion in the paper by A.~Erschler and the author \cite[\S~5.D]{Erschler-Kaimanovich23}). It's worth adding that the irreducibility condition (unlike non-degeneracy) is not restrictive, as it can always be satisfied by passing to a smaller group~$\gr\mu$.

The image of the random walk $(G,\mu)$ under a group epimorphism $\pi:G\mapsto \ol G$ is the random walk on the quotient group $\ol G$ with the step distribution $\ol\mu=\pi(\mu)$, and therefore the Poisson boundary $\p_{\ol\mu} \ol G$ is a quotient of the Poisson boundary $\Ps$ (more precisely, the space of ergodic components of the action of the normal subgroup $\ker\pi$). In particular, if the quotient measure $\ol\mu$ is not Liouville, then the original measure $\mu$ is not Liouville either.

\subsection{}

A function $f$ on a group $G$ is called \textsf{$\mu$-harmonic} if $Pf=f$ for the Markov operator $P$ \eqref{eq:p}. Because of the $\mu$-stationarity of the primary harmonic measure $\nu$, any function $\wh f\in L^\infty(\Ps,\nnu)$ gives rise to the $\mu$-harmonic function.
\begin{equation} \label{eq:poisson}
f(g)=\langle \wh f,g\nu\rangle \;.
\end{equation}
Furthermore, the \textsf{ Poisson formula} \eqref{eq:poisson} establishes an isometric isomorphism between the space $L^\infty(\Ps,\nnu)$ and the Banach space $H^\infty(G,\mu)$ of bounded $\mu$-harmonic functions on $G$ endowed with the $\sup$-norm. In this sense, the correspondence \eqref{eq:poisson} is an analogue of the classical Poisson formula for harmonic functions on the unit disk of the complex plane. In particular, the Poisson boundary $\Ps$ is trivial (i.e., is a singleton) if and only if there are no non-constant bounded $\mu$-harmonic functions on the group $G$ \textsf{(the Liouville property)}.

We call a measure $\mu$ on a group $G$ \textsf{Liouville} if the primary harmonic measure $\nu$ is a delta measure, which is equivalent to the absence of non-constant bounded $\mu$-harmonic functions on the semigroup $\sgr\mu$, i.e., on the set of points attainable by the random walk issued from the group identity. It is clear that the same is then true for any initial point. In probabilistic terms (see the Introduction), this means that there is no non-trivial stochastically significant behavior of the random walk at infinity for an arbitrary one-point initial distribution. The convenience of this definition of the Liouville property is that it allows one to drop the irreducibility assumption on the measure $\mu$. It is easy to see (e.g., using the Poisson formula) that the Liouville property for the measure $\mu$ implies the absence of non-constant bounded $\mu$-harmonic functions on the group $\gr\mu$ as well, and therefore  triviality of the Poisson boundary~$\Ps$ is equivalent to the combination of the Liouville property for the measure~$\mu$ with its irreducibility.
In the general case, the Poisson boundary of a Liouville measure $\mu$ is just the homogeneous space of the left cosets of the subgroup~$\gr\mu$.

\subsection{} \label{sec:im}

In more analytical terms, the Liouville property of the measure $\mu$ is equivalent to the strong convergence of the Ces\`aro means
\begin{equation} \label{eq:sit}
\si_t = \frac1t (\mu + \mu^{*2} + \dots + \mu^{*t} )
\end{equation}
of convolutions of the measure $\mu$ to left invariance on the group $\gr\mu$, i.e., to the left \textsf{asymptotic invariance}
\begin{equation} \label{eq:im}
\|g\si_t - \si_t \| \to 0 \qquad \forall\,g\in\gr\mu
\end{equation}
(here $\|\cdot\|$ denotes the total variation, in other words, the $\ell^1$ norm). In particular, if the measure $\mu$ is Liouville, then the group $\gr\mu$ is amenable. Instead of Ces\`aro means, one can also consider the measures obtained by averaging the convolutions $\mu^{*t}$ with the help of any other asymptotically invariant sequence of measures on $\ZZ_+$, for example, the binomial distributions $\bigl(\frac12(\de_0+\de_1)\bigr)^{*t}$, i.e., the sequence of one-dimensional distributions $\mathring{\mu}^{*t}$ of the ``lazy'' random walk with the transition measure $\mathring\mu=\tfrac12(\de_e + \mu)$ which is the half-sum of the initial measure $\mu$ and the delta measure at the identity element, see \cite[Theorem~2.3]{Kaimanovich92}. By passing from asymptotically invariant sequences of measures to invariant means, one can then say that the Liouville property of the measure $\mu$ is equivalent to the fact that some (equivalently, any) Banach limit of the sequence of its convolutions is a left-invariant mean on the group $\gr\mu$.

\begin{rem} \label{rem:lr}
It would be interesting to understand the structure of the set of means on a given group $G$ arising as the Banach limits of the sequences of convolutions $\mu^{*t}$ of a certain irreducible measure $\mu$. This question is of interest even in the situation when the measure~$\mu$ is not Liouville (in particular, when the group $G$ is non-amenable), see the work of A.~Fisher and the author \cite{Kaimanovich-Fisher98}. In such cases, the corresponding means are not invariant but only $\mu$-stationary (for this class, see the work of H. Furstenberg and E. Glasner \cite{Furstenberg-Glasner10}).
\end{rem}

In the opposite direction, the \emph{existence} of an irreducible Liouville measure on any amenable group was established by A. M. Vershik and the author \cite{Vershik-Kaimanovich79r, Kaimanovich-Vershik83} and independently by J.~Rosenblatt \cite{Rosenblatt81}. On the other hand, the characterization of groups on which \emph{every} measure is Liouville (these are the groups that coincide with their hyper-FC-center) was recently completed by J. Frisch, Y. Hartman, O. Tamuz, and P. Vahidi Ferdowsi \cite{Frisch-Hartman-Tamuz-Ferdowsi19}, who proved the existence of non-Liouville non-degenerate measures on all groups without this property (see also further results in this direction in the work of A. Erschler and the author \cite{Erschler-Kaimanovich23}).

\subsection{}

Let us note for further reference another property of Liouville measures, directly derived from their characterization in terms of asymptotic invariance \eqref{eq:im}.

\begin{clm} \label{clm:prod}
The Liouville property for the product measure $\mu\otimes\mu'$ on the product of groups $G \times G'$ is equivalent to both measures-multipliers $\mu$ and $\mu'$ being Liouville simultaneously.
\end{clm}

\begin{rem}
This statement is a special case of the fact that for the product $(G \times G', \mu \otimes \mu')$ of the random walks $(G,\mu)$ and $(G',\mu')$, the primary harmonic measure is the product $\nu \otimes \mu'$ of the primary harmonic measures of each multiplier, which in turn follows from the general result that the tail boundary of the product of two Markov chains is the product of their tail boundaries. Regarding the Poisson boundaries (which are endowed with the harmonic measure classes determined by initial distributions with full support), in general, $\p_{\mu \otimes \mu'} (G \times G') \cong \Ps \times \p_{\mu'} G'$ only if all measures $\mu, \mu', \mu \otimes \mu'$ are non-degenerate. A detailed discussion of the relationship between the Poisson boundary and the tail boundary for general Markov chains and their triviality is contained in \cite{Kaimanovich92}; their coincidence with respect to a single point initial distribution for random walks on groups is a cornerstone of the entropy theory of random walks, see \secref{sec:e}.
\end{rem}

For further information on the Poisson boundary of random walks on groups and of general Markov chains with discrete state spaces, we refer to the papers by A. M. Vershik and the author \cite{Kaimanovich-Vershik83, Kaimanovich92, Kaimanovich00a} and the literature cited therein, as well as to more recent surveys by A. Furman \cite{Furman02}, A. Erschler \cite{Erschler10}, and T. Zheng \cite{Zheng21p}.

\section{Entropy of random walks} \label{sec:e}

\begin{flushright}
\begin{tabular}{l}
\textcyr{Все проходит. Все не вечно.}
\\
\textcyr{Энтропия, друг ты мой!}
\\
\hfill\textit{\textcyr{Т.~Кибиров. Л.~С.~Рубинштейну}}
\end{tabular}
\end{flushright}

\subsection{}

Let us now consider the situation where the step distribution $\mu$ of a random walk on the group $G$ has finite \textsf{entropy}
$$
H(\mu) = - \sum_g \mu(g)\log\mu(g) \;.
$$
In this case the entropies of convolutions of the measure $\mu$ with itself are also finite, and their sequence is subadditive. The resulting limit
$$
h(\mu) = \lim_{t\to\infty} \frac{H(\mu^{*t})}{t}
$$
is called the \textsf{asymptotic entropy} of the measure $\mu$.

\begin{thm}[Vershik -- Kaimanovich \cite{Vershik-Kaimanovich79r,Kaimanovich-Vershik83}, Derriennic \cite{Derriennic80}] \label{thm:e}
A measure $\mu$ with finite entropy $H(\mu)$ is Liouville if and only if $h(\mu)=0$.
\end{thm}

\subsection{}

Let us recall that the \textsf{opposite (reflected, reversed) measure} $\check\mu$ is the image of a measure~$\mu$ on a group $G$ under the operation of group inversion:
\begin{equation} \label{eq:inv}
\check\mu(g) = \mu\left( g^{-1} \right) \quad \forall\,g\in G \;.
\end{equation}
The random walk $(G,\check\mu)$ is the reversal of the original random walk with respect to the invariant counting measure $\N$. Obviously, the entropies $H(\mu)$ and $H(\check\mu)$ of the original and of the reversed measures always coincide. Furthermore, since the inversion $\mu \mapsto \check\mu$ is an anti-automorphism of the group $\ell^1$ algebra, the entropies of the respective convolutions are also equal, and therefore, the asymptotic entropies $h(\mu)$ and $h(\check\mu)$ coincide as well. Thus, \thmref{thm:e} leads to a rather unexpected property.

\begin{prp}[\cite{Kaimanovich83ar,Kaimanovich-Vershik83}] \label{prp:1}
If $H(\mu)<\infty$, then the measure $\mu$ and the reflected measure $\check\mu$ are either both Liouville or both non-Liouville.
\end{prp}

In light of the discussion from \secref{sec:1}, the Liouville proverty of the reflected measure $\check\mu$ is equivalent to the left asymptotic invariance of the Ces\`aro averages of the convolutions $\check\mu^{*t}$ with respect to the group $\gr\check\mu=\gr\mu$, or equivalently, to the right asymptotic invariance of the Ces\`aro averages $\si_t$ \eqref{eq:sit} of the convolutions of the original measure $\mu$. Thus, we have the following:

\begin{cor}[\cite{Kaimanovich83ar,Kaimanovich-Vershik83}]
If $H(\mu)<\infty$, then the following properties of asymptotic invariance of the Ces\`aro averages of the convolution powers of the measure $\mu$ with respect to the action of the group $\gr\mu$ are equivalent:
\begin{itemize}
\item[(i)]
left asymptotic invariance;
\item[(ii)]
right asymptotic invariance;
\item[(iii)]
bilateral (i.e., both left and right) asymptotic invariance.
\end{itemize}
\end{cor}

\begin{rem} \label{rem:inv}
This statement can also be reformulated directly in terms of the means on the group $\gr\mu$ obtained by applying Banach limits to the sequence of convolutions $\mu^{*t}$ (see \secref{sec:im}): if $H(\mu)<\infty$, then any such mean is either invariant on both sides or not invariant on either side.
\end{rem}

\subsection{}

Another interesting consequence of the entropy theory is related to random walks on group products (cf. \clmref{clm:prod}). Let $\wt\mu$ be a probability measure on the product $\wt G = G \times G'$ of two countable groups (for simplicity, we consider the case of two multipliers only). Denote by $\mu$ and $\mu'$ its marginal distributions, i.e., the results of applying to the measure $\wt\mu$ the coordinate projections onto the groups $G$ and $G'$, respectively. Then, as is well known (for example, see V. A. Rokhlin's article \cite[\S~4.3 and \S~4.8]{Rohlin67r}),
\begin{equation} \label{eq:Hi}
H(\mu), H(\mu') \le H(\widetilde{\mu}) \le H(\mu) + H(\mu') \;,
\end{equation}
and the same inequalities hold for the corresponding convolution powers as well. Consequently,
\begin{equation} \label{eq:hi}
h(G,\mu), h(G',\mu') \le h(\widetilde{G},\widetilde{\mu}) \le h(G,\mu) + h(G',\mu') \;.
\end{equation}
By the entropy criterion from \thmref{thm:e}, we then have:

\begin{prp} \label{prp:2}
If the probability measure $\wt\mu$ on the product $\wt G=G\times G'$ of two groups $G$ and~$G'$ has finite entropy $H(\wt\mu)$, then it is Liouville if and only if both its marginal distributions are Liouville.
\end{prp}

\begin{rem} \label{rem:gen}
This proposition (essentially with the same proof) holds in all other situations where the entropy theory is applicable, namely when there is a natural probability measure preserving transformation on the corresponding path space. This is the case, for example, for random walks in random environment \cite{Kaimanovich90a}, on equivalence relations \cite{Kaimanovich-Sobieczky12}, or in the most general setting of invariant Markov chains with a finite stationary measure on groupoids with countable orbits \cite{Kaimanovich05a}. By considering conditional random walks or Poisson extensions (bundles) of the corresponding groupoids \cite[\S~5.A]{Kaimanovich05a}, it follows that, assuming finiteness of the entropy of the measure $\wt\mu$, the Poisson boundary
$\p_{\wt\mu}\wt G$ is fully described by the boundary behavior of the two quotient walks $(G,\mu)$ and $(G',\mu')$. The latter result can also be obtained by a direct application of author's entropy criterion for identification of the Poisson boundary \cite{Kaimanovich00a}, as it was done by A. Erschler and J. Frisch in \cite[Claim~4.1]{Erschler-Frisch22p}.
\end{rem}

\begin{rem}
It is well known (for example, \cite[\S~4.8]{Rohlin67r}) that the right-hand side inequality in~\eqref{eq:Hi} holds as an equality if and only if the measure $\wt\mu$ is the product $\mu\otimes\mu'$ of its marginals. It is natural to ask the same question about the additivity of asymptotic entropy, i.e., about the conditions under which the right-hand side inequality in \eqref{eq:hi} becomes an equality:
$$
h(\widetilde{G},\widetilde{\mu}) = h(G,\mu) + h(G',\mu') \;.
$$
It is reasonable to assume that this property should be related to the possibility of decomposing the primary harmonic measure of the random walk $(\wt G,\wt\mu)$ (or, in a weaker form, its measure class) into the product of the primary harmonic measures of the quotient random walks $(G,\mu)$ and $(G',\mu')$ (respectively, of their measure classes).
\end{rem}

Our goal is to present examples showing that \prpref{prp:1} and \prpref{prp:2} are no longer valid when the entropy finiteness assumption is dropped.

\section{Wreath products, dynamical configurations, etc.} \label{sec:3}

\begin{flushright}
\begin{tabular}{l}
\textcyr{Послушайте!}
\\
\textcyr{Ведь, если звезды зажигают --}
\\
\textcyr{значит — это кому-нибудь нужно?}
\\
\hfill \textit{\textcyr{В.~Маяковский. Послушайте!}}
\end{tabular}
\end{flushright}

\subsection{}

Let us denote by
$$
\Phi = \bigoplus_{-\infty}^\infty \ZZ_2 = \fun(\ZZ,\ZZ_2)
$$
the countable direct sum of copies of the two-element group $\ZZ_2=\{0,1\}$ (addition $\!\mod 2$) indexed by integers, i.e., the group of functions (\textsf{configurations}) $\ph:\ZZ\to\ZZ_2$ with finite \textsf{support}
$$
\supp\ph=\ph^{-1}(1)
$$
with respect to the pointwise addition operation. The group $\Phi$ is generated by the one-point configurations (``$\de$-functions'') $\eps_n$ concentrated at the points $n\in\ZZ$, and its identity is the \textsf{empty configuration} $\upth$, i.e., the function identically equal to zero. The shifts
\begin{equation} \label{eq:tn}
(T^n\ph)(z)=\ph(z-n) \;, \qquad n\in\ZZ \;,
\end{equation}
are automorphisms of $\Phi$, and
$$
T^n\eps_m=\eps_{n+m} \qquad\forall\,n,m\in\ZZ \;.
$$
Of course, the support of the configuration $T^n\ph$ is the shift of the support of $\ph$:
\begin{equation} \label{eq:supp}
\supp T^n \ph = n + \supp\ph \qquad\forall\, n\in\ZZ,\; \ph\in\Phi \;.
\end{equation}

The resulting semidirect product
\begin{equation} \label{eq:lc}
\Lc = \ZZ \sd \Phi = \ZZ \wr \ZZ_2
\end{equation}
is called the \textsf{direct (restricted) wreath product} of the active group $\ZZ$ with the passive group~$\ZZ_2$ (in contrast to the group-theoretic custom, in our notation the active group stands on the left). The group operation in $\Lc$ is
$$
(n_1,\ph_1)(n_2,\ph_2) = (n_1+n_2, \ph_1 + T^{n_1}\ph_2) \;.
$$
We will assume that the groups $\ZZ$ and $\Phi$ are embedded into $\Lc$ by means of the mappings
$$
n\mapsto (n,\upth) \;, \qquad \ph\mapsto (0,\ph) \;.
$$
Thus, the representation $(n,\ph)$ of elements of the group $\Lc$ determines its unique decomposition into the product $\Phi\ZZ$, the subgroup $\Phi$ is a normal subgroup of $\Lc$, and the resulting homomorphism takes the form
$$
\Lc \to \ZZ\cong\Lc/\Phi \;, \qquad (n,\ph) \mapsto n \;.
$$
It is also worth noting that the group $\Lc$ is generated by the generator of the group $\ZZ$ and the singleton configuration at $0$
\begin{equation} \label{eq:ep}
\eps=\eps_0\in\Phi \;.
\end{equation}

\begin{rem}
The history of wreath products dates back almost a hundred years (see the survey of C. Wells \cite[Section 16]{Wells76}), and the modern terminology and notation originate from P. Hall \cite[Section 1.4]{Hall54}. As for functional and stochastic analysis on groups, the idea of using wreath products was first suggested by A. M. Vershik, who considered in \cite{Vershik73r} the growth of F{\o}lner sets for the group $\ZZ\wr\ZZ$. The family of wreath products $\ZZ^d\wr\ZZ_2$ was then used by A. M. Vershik and the author \cite{Vershik-Kaimanovich79r, Kaimanovich83ar, Kaimanovich-Vershik83} to construct a number of new examples in the theory of random walks on groups. Although the ``lamp'' interpretation of wreath products with the passive group $\ZZ_2$ was mentioned by A. M. Vershik on numerous occasions at that time, a more formal name \emph{group of dynamic configurations} was chosen in \cite{Kaimanovich-Vershik83}.

Starting from the 1990s, the concise and expressive term \emph{lamplighter groups} has been used in the English language literature (apparently, for the first time by R. Lyons, R. Pemantle, and Y. Peres \cite{Lyons-Pemantle-Peres96a}): the components $n\in\ZZ$ and $\ph\in\Phi$ of the element $(n,\ph)$ of the group $\Lc$ describe, respectively, the position of the lamplighter and the configuration of the lit lamps on the integer line $\ZZ$. The Russian terminology has not stabilized; among the proposed options are \textcyr{\emph{лампочная группа}, \emph{группа мигающих лампочек}} and \textcyr{\emph{группа фонарщика}} (cf. the somewhat cumbersome French term \emph{groupes de l'allumeur de réverbères}).
\end{rem}

\subsection{}

Let us denote by
$$
\wt\Lc = \Lc\times\Lc
$$
the product of the group $\Lc$ by itself. Below we will need the ``semi-diagonal'' subgroup $\wt\Lc$ with a common $\ZZ$-component, i.e., the semidirect product
\begin{equation} \label{eq:l0}
\wt\Lc_0 = \ZZ \sd (\Phi\times\Phi) \;,
\end{equation}
determined by the diagonal action of the group $\ZZ$ on $\Phi\times\Phi$ by shifts \eqref{eq:tn}. The elements of the group $\wt\Lc_0$ are triples $(n,\ph,\ph')$, and the group operation takes the form
$$
(n_1,\ph_1,\ph'_1)(n_2,\ph_2,\ph'_2) = (n_1+n_2, \ph_1 + T^{n_1}\ph_2, \ph'_1+T^{n_1}\ph'_2) \;.
$$
In addition to the natural ``coordinate'' homomorphisms
\begin{equation} \label{eq:pi}
\pi: \wt\Lc_0 \to \Lc \;, \qquad \pi(n,\ph,\ph') = (n,\ph)
\end{equation}
and
\begin{equation} \label{eq:pi1}
\pi': \wt\Lc_0 \to \Lc \;, \qquad \pi(n,\ph,\ph') = (n,\ph') \;,
\end{equation}
there is also a ``combined'' or ``difference'' homomorphism
\begin{equation} \label{eq:pid}
\ol\pi: \wt\Lc_0 \to \Lc \;, \qquad \ol\pi(n,\ph,\ph')=(n,\ph+\ph')
\end{equation}
(we recall that all elements of the group $\Phi$ have order 2, so $\ph+\ph'=\ph-\ph'$ for any two configurations $\ph,\ph'\in\Phi$).

\pagebreak

\section{Random walks on the group $\Lc$} \label{sec:4}

\begin{flushright}
\begin{tabular}{l}
\textcyr{Мы люди неплохие,}
\\
\textcyr{Как вечер, мы в пути,}
\\
\textcyr{Фонарщики лихие,}
\\
\textcyr{Волшебники почти.}
\\
\hfill\textit{\textcyr{Б.~Окуджава. Песня фонарщиков}}
\end{tabular}
\end{flushright}

\subsection{}

Let us denote by
$$
\Phi_n^m=\langle \eps_i \rangle_{i=n}^m
$$
the finite subgroup of the group $\Phi$ consisting of those configurations whose support is contained in the integer interval $[n\,..\,m]$. We define the sequence
\begin{equation} \label{eq:kn}
\begin{aligned}
K_n &= \Phi_0^n \sm \Phi_0^{n-1}
=\eps_n + \Phi_0^{n-1} \\
&= \{\ph\in\Phi: \supp\,\ph\subset [0\,..\,n]\;\text{and}\; \ph(n)=1\} \;, \qquad n\ge 1 \;,
\end{aligned}
\end{equation}
of pairwise disjoint finite subsets of $\Phi$; note that
\begin{equation} \label{eq:2n}
\Phi_0^n=2^{n+1} \;, \qquad |K_n|=2^n \;.
\end{equation}
Let us denote by
\begin{equation} \label{eq:ka}
\ka_n=\Unif(K_n)
\end{equation}
the corresponding uniform probability measures on the sets $K_n\subset\Phi$. Now, let us fix a probability measure
\begin{equation} \label{eq:al}
\al=(\al_n)_{n\ge 1}
\end{equation}
on the set of positive integers; it gives rise to the probability measure
\begin{equation} \label{eq:la}
\la = \la_\al = \sum_{n=1}^\infty \al_n \ka_n
\end{equation}
on $\Phi$, which is the weighted average of the measures $\ka_n$ with weights $\al_n$. Finally, we define the probability measure
\begin{equation} \label{eq:mu}
\mu = \mu_\al = \de_{-1}\otimes\la \;.
\end{equation}
on the group $\Lc$.

By the definition of the measure $\mu$, the sample paths $\gb=(g_t)$ of the random walk $(\Lc,\mu)$ issued from the group identity have the form
\begin{equation} \label{eq:paths}
g_t = h_1h_2\dots h_t = (-1,\ph_1)(-1,\ph_2)\dots (-1,\ph_t) = (-t,\psi_t) \;,
\end{equation}
where $h_t=(-1,\ph_t)$ are independent $\mu$-distributed increments of the random walk, i.e., $(\ph_t)$ is a sequence of independent $\la$-distributed $\Phi$-valued random variables, and
\begin{equation} \label{eq:psi}
\psi_t = \ph_1 + T^{-1} \ph_2 + \dots + T^{-(t-1)} \ph_t \;.
\end{equation}

In other words, the movement of the lamplighter itself (i.e., of the $\ZZ$-component of the random walk) is deterministic, and at time $t$ its position is $-t$, while the configuration of the lamps (the $\Phi$-component $\psi_t$ of the random walk) is random. Namely, when passing from time $t$ to time $t+1$, the lamplighter samples a random configuration $\ph_{t+1}$ (or equivalently, a random subset $\supp\ph_{t+1}\subset\ZZ$) according to the distribution $\la$, and then changes the states of the lamps around him at the points whose coordinates relative to lamplighter's position lie in $\supp\ph_{t+1}$ (i.e., at the points whose absolute positions are described by the shift $T^{-t}\ph_{t+1}$), so that
\begin{equation} \label{eq:step}
\psi_{t+1} = \psi_t + T^{-t} \ph_{t+1} \;,
\end{equation}
after which the lamplighter himself moves by $-1$ and transitions from point $-t$ to point $-t-1$.

\subsection{}

Now let us look at the Liouville property for the measures described above.

\begin{clm} \label{clm:l1}
If the measure $\al$ \eqref{eq:al} has finite support, then the corresponding measure $\mu$ \eqref{eq:mu} on the group $\Lc$ is non-Liouville.
\end{clm}

\begin{proof}
Since the supports of all configurations $\ph_t$ in formula \eqref{eq:psi} are contained in the finite interval $[0\,..\,M]$ for some integer $M>0$, one can appy to the random walk $(\Lc,\mu)$ an argument stemming from the work of A. M. Vershik and the author \cite{Vershik-Kaimanovich79r, Kaimanovich83ar, Kaimanovich-Vershik83} (it connects transience of the projection of the walk onto the active group of the wreath product with the non-triviality of the Poisson boundary). Indeed, according to \eqref{eq:supp}, in this case
$$
\supp T^{-t} \ph_{t+1} \subset [-t\,..\,-t+M] \quad \forall\, t\ge 0 \;,
$$
and therefore the values $\psi_t(z)$ of the configurations $\psi_t$ \eqref{eq:psi} at any point $z\in\ZZ$ stabilize as $t\to\infty$. It is easy to verify that the resulting (infinite) configuration $\psi_\infty=\lim_t\psi_t$ is random (see \cite[Theorem 3.3]{Kaimanovich91}), and hence the primary harmonic measure of the random walk $(\Lc,\mu)$ is non-trivial.
\end{proof}

This reasoning does not hold in the case where the support $\supp\al$ is infinite. Below we will see that in this situation the measure $\mu$ \eqref{eq:mu} may also be Liouville, and in \lemref{lem:lem} we establish a simple necessary and sufficient condition for this. But before that, keeping in mind our goal of constructing counterexamples to \prpref{prp:1} and \prpref{prp:2} in the case of infinite entropy, let us take a look at two other closely related random walks.

\subsection{}

The first one is the random walk on the same group $\Lc$ determined by the reflection $\check\mu$ \eqref{eq:inv} of the measure $\mu$.

\begin{clm} \label{clm:reflected}
The reflected measure $\check\mu$ is non-Liouville for any measure $\al$ \eqref{eq:al} from the definition \eqref{eq:mu}.
\end{clm}

\begin{proof}
Since the inversion in the group $\Lc$ is given by
$$
(n,\ph)^{-1} = (-n, T^{-n}\ph ) \quad \forall\, (n,\ph)\in\Lc \;,
$$
one has
$$
\check\mu = \de_1 \otimes T\la \;,
$$
where the shift $T\la$ of the measure $\la$ \eqref{eq:la} only charges the configurations concentrated on the positive ray $[1\,..\,\infty)\subset\ZZ$. Therefore, the same reasoning as in the proof of \clmref{clm:l1} applies without any restrictions on the measure $\al$.
\end{proof}

\subsection{}

The state space of the second random walk is the group $\wt\Lc_0$ \eqref{eq:l0}. To define its step distribution, let us first take a sequence of probability measures $\wt\ka_n$ on the squares $K_n\times K_n$ of the sets $K_n$ \eqref{eq:kn} such that their marginal distributions (i.e., the images under the projections onto each of the copies of $K_n$) coincide with the corresponding measures $\ka_n$ \eqref{eq:ka}. Since the distributions $\ka_n$ are uniform, in other words, the measures $\wt\ka_n$ are exactly the result of normalizing bistochastic matrices parameterized by the sets $K_n$.

We are interested in the images
$$
\ol\ka_n=\om(\wt\ka_n)
$$
of the measures $\wt\ka_n$ under the action of the group operation in $\Phi$
\begin{equation} \label{eq:phh}
\om:(\ph,\ph')\mapsto \ph+\ph' \;, \qquad \ph,\ph'\in\Phi \;.
\end{equation}
It is clear that $\ol\ka_n=\de_\upth$ (we recall that the trivial ``empty'' configuration $\upth$ is the identity of the group $\Phi$) if and only if the measure $\wt\ka_n$ is concentrated on the diagonal of the product $K_n\times K_n$. Of course, there are other measures on $K_n\times K_n$ with the prescribed marginals~$\ka_n$; according to the Birkhoff -- von Neumann theorem, all of them are convex combinations of the extreme measures corresponding to permutations of the sets $K_n$. In particular, for $n=1$, the set $K_1$ consists of two configurations $\eps_1$ and $\eps_0+\eps_1$, and we can take for $\wt\ka_1$ the uniform distribution on the antidiagonal in $K_1\times K_1$
$$
\wt\ka_1 = \Unif \{ (\eps_1,\eps_0+\eps_1), (\eps_0+\eps_1,\eps_1) \} \;,
$$
that corresponds to the transposition of $\eps_1$ and $\eps_0+\eps_1$. With this choice, $\ol\ka_1=\om(\wt\ka_1)$ is the delta measure concentrated on the configuration $\eps=\eps_0$ \eqref{eq:ep}.

In general, it is obvious that
$$
\supp\ol\ka_n \subset K_n + K_n = \Phi_0^{n-1} \;,
$$
and it can be easily seen that for any prescribed measure $\rho$ on $\Phi_0^{n-1}$, we can choose a measure~$\wt\ka_n$ on the product $K_n\times K_n$ in such a way that its marginal distributions coincide with $\ka_n$, and the image $\ol\ka_n$ under the action of the mapping \eqref{eq:phh} is $\rho$. For instance, using the fact that the uniform distribution $m=\Unif(\Phi_0^{n-1})$ on $\Phi_0^{n-1}$ is preserved under convolution with any measure on this subgroup, we can take the image of the product measure $m\otimes\rho$ under the mapping
$$
(\ph_0,\ol\ph) \mapsto (\ph_0+\eps_n,\ph_0+\ol\ph+\eps_n) \;.
$$

Now, let us take the weighted average
$$
\wt\la = \sum_{n=1}^\infty \al_n \wt\ka_n
$$
with the same weights $\al_n$ as in the definition \eqref{eq:la} of the measure $\la$, and finally set
\begin{equation} \label{eq:wtmu}
\wt\mu = \de_{-1} \otimes \wt\la \;.
\end{equation}

\begin{clm} \label{clm:wtmu}
If the support of the distribution $\al$ \eqref{eq:al} is infinite, and the set of indices $n\in\supp\al$ for which $\wt\ka_n\neq \ka_n\otimes\ka_n$ is non-empty and finite, then the measure $\wt\mu$ is non-Liouville.
\end{clm}

\begin{proof}
By definition, the image of the random walk $(\wt\Lc_0,\wt\mu)$ under the homomorphism \eqref{eq:pid} is the random walk on the group $\Lc$ with the step distribution
$$
\ol\mu = \de_{-1} \otimes \ol\la \;,
$$
where
$$
\ol\la = \sum_{n=1}^\infty \al_n \ol\ka_n \;.
$$
Thus, to prove that the measure $\wt\mu$ is not Liouville, it suffices to establish that the quotient measure $\ol\mu$ is not Liouville. By assumption, the support of the measure $\ol\la$ is finite and contains at least one configuration other than the empty configuration $\upth$, since $\supp\ol\ka_n=\{\upth\}$ if and only if $\wt\ka_n= \ka_n\otimes\ka_n$. Therefore, the measure $\ol\mu$ is non-Liouville for the same reasons as in the proof of \clmref{clm:l1}.
\end{proof}

\section{Infinite entropy} \label{sec:inf}

\begin{flushright}
\begin{tabular}{l}
\dots so ist das Endliche nur als das,
\\
wor\"uber hinausgegangen werden mu\ss,
\\
die Negation seiner an ihm selbst,
\\
welche die Unendlichkeit ist.
\\
\hfill\textit{G.~W.~F.~Hegel. Wissenschaft der Logik}
\end{tabular}
\end{flushright}

\subsection{}

The key technical result is the following.

\begin{lem} \label{lem:lem}
The measure $\mu$ \eqref{eq:mu} on the group $\Lc$ \eqref{eq:lc} is Liouville if and only if the entropy $H(\mu)$ is infinite.
\end{lem}

\begin{proof}
Since the reflected measure $\check\mu$ is always non-Liouville (\clmref{clm:reflected}), if the entropy $H(\mu)$ is finite, then by \prpref{prp:1} the measure $\mu$ itself is also non-Liouville.

Now, let us consider the situation when the entropy $H(\mu)$ is infinite and show that in this case the measure $\mu$ is Liouville. First of all, note that since $|K_n|=2^n$, the entropy of the measure $\mu$ is
$$
H(\mu) = H(\al) + (\log 2)|\al| \;,
$$
where
$$
|\al| = \sum_n n \al_n
$$
denotes the first moment of the measure $\al$ \eqref{eq:al}. On the other hand, $H(\al)\le (\log 2)|\al|$ due to the Gibbs inequality applied to the measures $\al$ and $\xi=(\xi_n)_{n\ge 1}$, where $\xi_n=2^{-n}$. Thus, the entropy $H(\mu)$ is infinite if and only if the first moment $|\al|$ of the measure $\al$ is infinite.

An immediate consequence of this fact is the following observation. Let us denote by
$$
|\ph| = \max \{ |n|: n\in\supp\ph \}
$$
the ``range'' of the configuration $\ph\in\Phi$, also putting $|\upth|=0$ for the empty configuration $\upth$. In particular, $|\ph|=n$ for all configurations $\ph$ from any set $K_n$ \eqref{eq:kn}, and therefore the first moment $|\al|$ is exactly the expectation of the ranges $|\ph_n|$ of the increments of the random walk \eqref{eq:paths}. If it is infinite, then a routine application of the Borel -- Cantelli lemma implies that almost surely
\begin{equation} \label{eq:limsup}
\limsup_{t\to\infty} \bigl( |\ph_t|-t \bigr) = \infty \;.
\end{equation}

By the definition of the measure $\mu$, its convolutions (i.e., the one-dimensional distributions of the associated random walk) have the form
$$
\mu^{*t} = \de_{-t} \otimes \la_t \;,
$$
where $\la_t$ are the probability measures on the group $\Phi$ representing the distributions of the component $\psi_t$ \eqref{eq:psi} of the random walk. In other words, they are the convolutions
$$
\la_t = \la * T^{-1}\la * \dots * T^{-t+1}\la
$$
of the shifts of the measure $\la$ \eqref{eq:la} on the commutative group $\Phi$, and in particular
$$
\la_{t+1} = \la * T^{-1}\la_t \qquad\forall\,t\ge 1 \;.
$$
Therefore, as follows from the characterization of the Liouville property in terms of the asymptotic invariance of Ces\`aro means \eqref{eq:im}, for proving the Liouville property of the measure~$\mu$ it is sufficient to verify the asymptotic invariance property of the sequence of measures $(\la_t)$ with respect to the action of the group $\Phi$, namely,
\begin{equation} \label{eq:conv}
\| \ph \la_t - \la_t \| \xrightarrow[t\to\infty]{} 0 \qquad\forall\,\ph\in\Phi \;.
\end{equation}

The idea behind proving the convergence \eqref{eq:conv} is to replace one of the increments $\ph_\tau$ of each sample path $\gb$ \eqref{eq:paths} with
$$
\ph'_\tau=\ph_\tau + T^{\tau-1} \ph \;,
$$
leaving the other increments unchanged. As follows from formula \eqref{eq:psi}, for the resulting new sample paths
\begin{equation} \label{eq:prime}
\psi'_t = \psi_t + T^{-\tau+1} (\ph'_\tau - \ph_\tau) = \psi_t + \ph \qquad \forall\,t\ge\tau \;.
\end{equation}
If this replacement preserves the measure on the space of sample paths, or equivalently, on the space of sequences of independent increments $(\ph_t)$ with a common distribution $\la$ \eqref{eq:la}, then the corresponding one-dimensional distributions $\la_t$ will have the desired property of asymptotic invariance \eqref{eq:conv}.

To implement this idea, let us fix a non-empty configuration $\ph\in\Phi$, and, by using the fact that the limit \eqref{eq:limsup} is infinite, for almost every sample path $\gb$ put
\begin{equation} \label{eq:tau}
\tau = \tau(\gb) = \min \{ t> |\ph|: |\ph_t|-t> |\ph|  \} \;.
\end{equation}
Now let us define the transformation
$$
U=U_\ph:(\ph_t)\mapsto(\ph'_t)
$$
on the space of sequences of independent $\la$-distributed increments $(\ph_t)$ by putting
$$
\ph'_t =
\begin{cases}
\ph_t+ T^{\tau-1}\ph  \;, & t=\tau \;, \\
\ph_t \;, & t\neq\tau \;.
\end{cases}
$$
The conditional distribution of $\ph_\tau$, given the value $\tau=t_0$, is a convex combination of the measures $\ka_n$ \eqref{eq:ka} with the weights obtained by normalizing the restriction of the distribution~$\al$~\eqref{eq:al} to the ray $(|\ph|+t_0\,..\,\infty)$, and therefore it is invariant under the action of the subgroup $\Phi_0^{|\ph|+t_0}$. On the other hand, the support of the configuration $\ph$ is contained in the interval $[-|\ph|\,..\,|\ph|]$ by the definition of the range $|\ph|$, and therefore, due to the inequalities from definition \eqref{eq:tau}, the support of the shift $T^{t_0-1}\ph$ is contained in the interval $[0\,..\,|\ph|+\tau_0]$, implying that $T^{t_0-1}\ph\in \Phi_0^{|\ph|+t_0}$.

As follows from formula \eqref{eq:prime}, for any $t_0>0$
$$
\| \ph\la_t - \la_t \| \le 2 \Pb\{\tau > t_0 \} \qquad\forall\,t\ge t_0 \;,
$$
which implies the convergence \eqref{eq:conv} and completes the proof.
\end{proof}

\subsection{}

\lemref{lem:lem} makes it easy to proceed with the presentation of our main examples.

\begin{thm} \label{thm:2}
If the measure $\mu$ \eqref{eq:mu} on the group $\Lc$ \eqref{eq:lc} has infinite entropy $H(\mu)$, then the reflected measure $\check\mu$ is non-Liouville, while the measure $\mu$ itself is Liouville.
\end{thm}

\begin{proof}
This is a combination of \clmref{clm:reflected} and \lemref{lem:lem}.
\end{proof}

\begin{rem} \label{rem:im2}
In light of \remref{rem:inv}, \thmref{thm:2} implies that if the entropy of the measure~$\mu$ is infinite, then all means on the group $\gr\mu$, obtained as Banach limits of the sequence of convolutions $\mu^{*t}$, are left-invariant but not right-invariant. Further elaborating the question from \remref{rem:lr}, it would be interesting to understand the structure of the set of the strictly one-sided means that arise in this way from convolutions of an irreducible measure $\mu$ on an arbitrary amenable group $G$ (let us denote this set $\Sc(G)$).

The question of describing the class of groups on which all invariant means (not just those arising from convolutions) are two-sided was first posed by G. M. Adelson-Velsky and Yu.~A.~Shreider \cite[Theorem 7]{Adelson-Shreider57r} (note in parentheses that they were not familiar with any of the works on invariant means that appeared after the publication of the Banach -- Tarski paradox, including von Neumann's 1929 paper), who, however, erroneously claimed that all groups of locally subexponential growth belong to this class. This error was discovered by A.~Paterson \cite{Paterson79}, who proved that in fact this class is much smaller and consists (if we restrict ourselves to discrete groups only) of FC-central groups, i.e., those in which all conjugacy classes are finite (see also further works by P. Milnes \cite{Milnes81}, J. Rosenblatt~-- M.~Talagrand~\cite{Rosenblatt-Talagrand81}, and J. Hopfensperger \cite{Hopfensperger20}). Thus, for all groups $G$ that are hyper-FC-central (see the discussion in \secref{sec:1}), but not FC-central, the set $\Sc(G)$ is empty. On the other hand, a recent result of A. V. Alpeev \cite{Alpeev21p} implies that $\Sc(G)$ is non-empty for all other amenable groups $G$.
\end{rem}

\begin{thm} \label{thm:3}
If the measure $\wt\mu$ \eqref{eq:wtmu} on the group $\wt\Lc_0 \subset \Lc\times\Lc$ \eqref{eq:l0} has the property that the measure $\al$ \eqref{eq:al} has infinite first moment, and the set of indices $n\in\supp\al$ with $\wt\ka_n\neq \ka_n\otimes\ka_n$ is non-empty and finite, then the measure $\wt\mu$ is non-Liouville, while both coordinate projections of the measure $\wt\mu$ onto the multipliers of the product $\Lc\times\Lc$ are Liouville.
\end{thm}

\begin{proof}
This is a combination of \clmref{clm:wtmu} and \lemref{lem:lem}.
\end{proof}

\section{Appendix. ``Liouville's theorem'' and ``Shannon's theorem''}

\begin{flushright}
\begin{tabular}{l}
\textcyr{Нет уж извините, не допущу}
\\
\textcyr{пройти позади такому приятному,}
\\
\textcyr{образованному гостю.}
\\
\hfill\textit{\textcyr{Н.~Гоголь. Мертвые души}}
\end{tabular}
\end{flushright}

\subsection{}

The events related to the emergence of what is now called Liouville's theorem on the absence of non-constant bounded analytic (as well as harmonic) functions on the complex plane are very well documented in the reports \emph{(Comptes Rendus)} of the meetings of \emph{l'Acad\'emie des sciences} for the years 1844 and 1851, and even in a brief account they are quite interesting (we omit numerous even more picturesque details, which we hope to return to on another occasion).

At the meeting on December 9, 1844, Liouville, in his remarks on a memoir on elliptic functions presented that day by Chasles, announces, among other things, his result on the absence of non-constant bounded doubly periodic analytic functions on the complex plane (referred to by him as a ``principle,'' because, according to Liouville, it \emph{``para\^it imprimer \`a l'\'etude de ces fonctions un caract\`ere d'unit\'e et de simplicit\'e tout particulier''}). Already at the next meeting on December 16 (exactly a week later), Cauchy asserts that Liouville's principle is a special case of his own results, published as early as in 1827, and presents a complete (by the standards of that time) proof without any assumptions of periodicity (this is what we currently know as ``Liouville's theorem''). A week later (December 23), Cauchy presents another proof, and altogether, according to U. Bottazzini's calculations \cite[p.~165]{Bottazzini86b}, Cauchy publishes five different proofs in less than a year.

The second act takes place on March 31, 1851. As the chairman of a special \emph{ad hoc} committee, Cauchy presents to the Academy a report on a memoir by Hermite on doubly periodic functions. At the very end of this report, Liouville's announcement is mentioned (the proof still has not been published). Liouville immediately informs that a complete proof is contained in the notes of lectures he himself gave in 1847 to two German mathematicians and recorded by one of them, Borchardt \emph{(or j'ai chez moi, et je pourrai déposer sur le bureau, avant la fin de la séance, une pièce manuscrite qui paraîtra concluante à cet égard)}, and indeed manages to present this manuscript before the end of the session (it was published by Borchardt much later, only in 1879). The table of contents of these lecture notes reproduced in \emph{Comptes Rendus} begins with the theorem on the absence of non-constant bounded doubly periodic functions. But the last word on this day belongs to Cauchy, who once again asserts his priority, referring not only to the note of 1844, but also to his earlier works.

The presentation of Liouville's theorem in \emph{``A Course of Modern Analysis''} by Whittaker and Watson is accompanied with a note (starting from the second edition of 1915): ``This theorem, which is really due to Cauchy, was given this name by Borchardt, who heard it in Liouville's lectures in 1847.'' This story has been discussed many times by historians of mathematics (in addition to the aforementioned book \cite{Bottazzini86b}, see also detailed analyses by J.~L\"utzen \cite[\S\S~XIII.9--17,~21]{Lutzen90b} and U. Bottazzini -- J. Gray \cite[\S\S~3.5.6,~4.2.4]{Bottazzini-Gray13b}), who ultimately tend to give priority to Liouville (\cite[pp.~543--544]{Lutzen90b} and \cite[pp.~231--232]{Bottazzini-Gray13b}, respectively), referring, among other things, to the note found in his records in the summer of 1844, mentioning how the general case can be obtained from the doubly periodic case. And yet, as for the general case, Cauchy's reasoning appears much more conceptual from a modern point of view (since in modern terms the doubly periodic case is ``just'' an exercise in applying the maximum principle \ldots).

Ever since Bernhard Riemann ``virtually put equality signs between two-dimensional potential theory and complex function theory,'' to quote L. Ahlfors \cite{Ahlfors53}, the statement about the absence of non-constant bounded harmonic functions on the plane has also been called Liouville's theorem. We will continue the story of the Liouville property for harmonic functions elsewhere.

\subsection{}

The definition of the entropy
\begin{equation} \label{eq:es}
H(p) = - \sum p_i \log p_i
\end{equation}
of a discrete probability distribution $p=(p_i)$ is usually associated with the name of Claude Shannon. However, in his seminal paper ``A Mathematical Theory of Communication,'' he explicitly writes that ``entropy appears in the same form in certain formulations of statistical mechanics'' \cite[\S~6]{Shannon48}. There is a widely circulated, albeit unconfirmed by Shannon himself, anecdote about his conversation with John von Neumann, who allegedly suggested using the term ``entropy,'' saying: ``You should call it entropy, for two reasons. In the first place your uncertainty function has been used in statistical mechanics under that name. In the second place, and more importantly, no one knows what entropy really is, so in a debate you will always have the advantage'' (for instance, see O. Rioul \cite[\S~12]{Rioul21}).

When discussing the genesis of the entropy \eqref{eq:es} from statistical mechanics, the first names that come to mind are Ludwig Boltzmann and Josiah Willard Gibbs (for the entropy in classical and quantum mechanics, see the very informative survey by S. Goldstein, J. Lebowitz, R. Tumulka, and N. Zanghi \cite{Goldstein-Lebowitz-Tumulka-Zanghi20}). Boltzmann provided a probabilistic interpretation of Clausius' thermodynamic entropy as the negative logarithm of the ``probability'' that the system is in a given state. The use of quotation marks is explained by the fact that in reality this quantity (the measure of permutability or \emph{Permutabilit{\" a}tsma\ss} according to Boltzmann \cite[p.~192]{Boltzmann77}; in modern terms, this is the differential entropy of a distribution in the phase space) arises as a result of some limiting procedure based on combinatorial calculations with infinitesimal regions, and therefore it is defined (even without demanding excessive rigor from the limiting process) up to an additive constant. Boltzmann himself repeatedly emphasized that the discretization he used ``is nothing more than an artificial device \emph{(Hilfsmittel)} that helps to calculate physical processes,'' explicitly calling it a ``mathematical fiction'' \cite[p.~348]{Boltzmann72}.

Although the sums of the form \eqref{eq:es} often appear in the works of Boltzmann and (slightly later) Gibbs, their appearance always serves as an auxiliary step towards the differential entropy (``all the infinite in nature means nothing other than some limiting procedure,'' as Boltzmann says \cite[p.~167]{Boltzmann77}), and the normalization condition for the weights $\sum p_i=1$ is never imposed. For example, the proof of Theorem VIII in Gibbs's "Elementary Principles in Statistical Mechanics" \cite{Gibbs02} essentially boils down to establishing what is now called Gibbs's inequality (the non-negativity of the Kullback -- Leibler divergence in the discrete case), but Gibbs does not feel any need to discuss it as a property of probability distributions.

Max Planck poses and solves the problem of the physical meaning of Boltzmann's limiting procedure, which leads him to the understanding that quantization is a physical reality. He vividly recalls in his Nobel lecture of 1920 \cite{Planck75r} how, following the discovery of the law of black-body radiation, the necessity arose to give it a conceptual interpretation: ``But even if this radiation formula should prove to be absolutely accurate it would after all be only an interpolation formula found by happy guesswork, and would thus leave one rather unsatisfied. I was, therefore, from the day of its origination, occupied with the task of giving it a real physical meaning, and this question led me, along Boltzmann's line of thought, to the consideration of the relation between entropy and probability; until after some weeks of the most intense work of my life clearness began to dawn upon me, and an unexpected view revealed itself in the distance.''

The perspective that opened up was nothing other than the first glimpse of quantum theory. Later in the same lecture, Planck explains in detail why the new theory implies that physical entropy (not just its increments) has an ``absolute value,'' i.e., in modern terms, it is a function and not an additive cocycle. In an expanded form this explanation is contained in the beginning (\S\S~113--131) of the section ``Entropy and Probability'' of his book ``The Theory of Heat Radiation'' (starting from the second edition of 1913 \cite{Planck13b}). As he writes in the preface, ``the entropy of a state has a quite definite, positive value, which, as a minimum, becomes zero, while in contrast therewith the entropy may, according to the classical thermodynamics, decrease without limit to minus infinity. For the present, I would consider this proposition as the very quintessence of the hypothesis of quanta.'' Formula (173) in \S~124 of Planck's book
$$
S = - kN \sum w_n\log w_n \;,
$$
where $S$ is the physical entropy, $k$ is the Boltzmann constant introduced by Planck, and $N$ is the number of molecules, is apparently the first appearance of the ``mathematical'' entropy \eqref{eq:es} with probability weights $w_n$.

Returning to von Neumann, it is precisely this formula (and, of course, with an explicit use of the fact that the weights in it are normalized; in the quantum setup this means that the trace of the corresponding state is one) that he quotes when defining the quantum ``von Neumann entropy'' \cite{vonNeumann27}.

In conclusion, it should be emphasized that all the above applies only to the entropy of a single discrete probability distribution. Shannon introduced into consideration sequences of distributions the entropies of which are subadditive, making it possible to talk of the arising \textsf{asymptotic entropy}, i.e., the rate of linear growth of these entropies. The additional structure used in this case is a stationary sequence of symbols, and the resulting sequence of measures is its finite dimensional distributions. Another example is the sequence of convolutions described in \secref{sec:e}; it is generated by a group structure on a countable set of states (and it is precisely Shannon's asymptotic entropy that served as a model for defining the asymptotic entropy of random walks by A. Avez). The question of a unified approach to these two situations was raised back in the 1970s by A. M. Vershik; later \cite[p.~64]{Vershik00r}, he proposed to use Poisson boundary polymorphisms for this purpose. As for the Boltzmann -- Planck -- von Neumann entropy in statistical physics, as recently noted by Vershik \cite[p.~49]{Vershik10ar}, its connection with Shannon's theory remains problematic.

\end{document}